\documentclass{amsart}
\usepackage[utf8]{inputenc}
\usepackage{amsmath}
\usepackage{xcolor}

\newtheorem{theorem}{Theorem}[section]
\newtheorem{proposition}[theorem]{Proposition}
\newtheorem{lemma}[theorem]{Lemma}
\newtheorem{corollary}[theorem]{Corollary}

\title[Homogeneous involutions on upper triangular matrices]{Homogeneous involutions on upper triangular matrices}

\begin{document}
\author{Thiago Castilho de Mello}
\address{Universidade Federal de S\~ao Paulo, Instituto de Ci\^encia e Tecnologia, S\~ao Jos\'e dos Campos, Brazil}
\email{tcmello@unifesp.br}

\begin{abstract}
    Let $K$ be a field of characteristic different from 2 and let $G$ be a group. If the algebra $UT_n$ of $n\times n$ upper triangular matrices over $K$ is endowed with a $G$-grading $\Gamma: UT_n=\oplus_{g\in G}A_g$ we give necessary and sufficient conditions on $\Gamma$ that guarantees the existence of a homogeneous antiautomorphism on $A$, i.e., an antiautomorphism $\varphi$ satisfying $\varphi(A_g)=A_{\theta(g)}$ for some permutation $\theta$ of the support of the grading. It turns out that $UT_n$ admits a homogeneous antiautomorphism if and only if the reflection involution of $UT_n$ is homogeneous. Moreover, we prove that if one homogeneous antiautomorphism of $UT_n$ is defined by the map $\theta$ then any other homogeneous antiautomorphism is defined by the same map $\theta$.
   
\end{abstract}

\maketitle

\section*{Introduction}

In the last two decades, gradings on algebras have been intensively studied. One of the main purposes of such studies was to give a description of all gradings on a given algebra or class of algebras up to graded isomorphisms. As examples, gradings on finite dimensional simple associative, Lie and Jordan algebras have been studied since then (see \cite{BahZai, BahShZai, ElKoch}). 
The description of gradings on simple associative algebras, was of particular importance to the description of gradings on some simple Lie and Jordan Algebras. Also the description of gradings on non-simple important algebras have been carried out (see \cite{Yukihide}).

The special case of $\mathbb{Z}_2$-gradings on matrix algebras was applied to the study of identities of associative algebras, in particular the solution of the Specht problem was strongly based on it. Graded identities on associative and non-associative algebras has become of great interest of many authors since then and an extensive number of papers about this subject has been written.

Gradings on matrix algebras were applied to obtain gradings on simple finite dimensional Lie and Jordan algebras (see, for example \cite{BahShZai, BahZai, ElKoch}). By considering involutions on matrix algebras, gradings on the Lie algebra of skew-symmetric elements and on the Jordan algebra of symmetric elements were obtained. This was shown to be possible when such involutions act by preserving degrees of homogeneous elements, the so called \emph{graded involutions}.

In contrast with the above case, when considering the transpose involution on $M_n(K)$ with an elementary grading, one can see that it acts on homogeneous components by inverting degrees. Involutions satisfying this property on graded algebras are called \emph{degree-inverting} involutions. Such involutions arise naturally in other contexts, for instance, the usual involution on Leavitt path algebras endowed with its natural $\mathbb{Z}$-grading are degree-inverting involutions.

The graded involutions and the degree-inverting involutions on full and upper triangular matrix algebras have been completely described when $K$ is a field of characteristic different from 2 (see \cite{BahShZai, BahZai,  FondeM, FonSanYas, VaZai2}). Also, ordinary and graded involutions on block-triangular matrices were studied in \cite{FGDY}. 

A natural generalization after the description of graded and degree-inverting involutions is to describe the homogeneous involutions, i.e., involutions $*$ on $G$-graded algebras $\Gamma: A=\oplus_{g\in G}A_g$, for which there exists a bijection $\theta:Supp \Gamma \longrightarrow Supp\Gamma$ satisfying $(A_g)^*=A_{\theta(g)}$, for each $g\in Supp \Gamma$, where $Supp \Gamma$ denotes the support of the grading $\Gamma$.

In this paper we completely solve this problem for the algebra of upper triangular matrices, namely, we give necessary and sufficient conditions on the grading so that it admits a homogeneous involution. Actually, we do it for a more general class of maps: antiautomorphisms. 

The paper is organized as follows: In section 1, we present the definitions and basic results concerning gradings, automorphisms, antiautomorphisms and involutions on upper triangular matrices. Section 2 describes the homogeneous automorphisms of the algebra of upper triangular matrices. Finally, in section 3 we present the main results of the paper, describing conditions under with the existence of homogeneous involutions and antiautomorphisms is guaranteed.

\section{Preliminaries}

Unless otherwise stated, in this paper $G$ denotes a group with neutral element $e$ and $K$ denotes a field of characteristic different from 2. The word \emph{algebra} means an associative algebra over the field $K$. By $M_n(K)$ we denote the algebra of $n\times n$ matrices over $K$ and by $UT_n$ the algebra of $n\times n$ upper triangular matrices over $K$. For each $1\leq i, j \leq n$, we define the matrix $E_{i,j}$ as the matrix with 1 in entry $(i,j)$ and 0 elsewhere. We call such matrices \emph{matrix units}.

If $A$ is an algebra, a \emph{$G$-grading}, $\Gamma$ on $A$ is a decomposition of $A$ as a direct sum of subspaces indexed by elements of $G$, \[\Gamma: A=\bigoplus_{g\in G} A_g\]
such that for each $g, h\in G$, we have $A_gA_h\subseteq A_{gh}$. In this case, we say that $A$ is a $G$-graded algebra. If  $a\in A_g$ we say that $a$ is  \emph{homogeneous of degree $g$} and we denote $\deg a=g$. The set \[\textrm{Supp}\, \Gamma = \{g\in G\,|\, A_g\neq \{0\}\}\] is called the \emph{support} of the grading.

If $A$ and $B$ are $G$-graded algebras, a homomorphism of algebras $\varphi: A\longrightarrow B$ is a \emph{graded homomorphism} if for any $g\in G$, $\varphi(A_g)=B_g$. In particular, if $\varphi$ is bijective, we say that $\varphi$ is a \emph{graded isomorphism}. A \emph{graded automorphism} of a graded algebra $A$ is a graded isomorphism $\varphi: A\longrightarrow A$. 

If $A$ is an algebra, and $\theta:\textrm{Supp}\, \Gamma \longrightarrow \textrm{Supp}\, \Gamma$ is a bijection, we say that an automorphism $\varphi$ of $A$ is \emph{$\theta$-homogeneous} if $\varphi(A_g)=A_{\theta(g)}$, for any $g\in \textrm{Supp}\, \Gamma$. We say that $\varphi$ is homogeneous if there exists some $\theta$ such that $\varphi$ is $\theta$-homogeneous.

In this paper we are specially interested in antiautomorphisms and involutions. An \emph{antiautomorphism} on an algebra $A$ is a linear map $\varphi:A\longrightarrow A$ satisfying $\varphi(xy)=\varphi(y)\varphi(x)$, for any $x,y\in A$. An antiautomorphism $\varphi$ is an \emph{involution} if it has order 2, i.e., if $\varphi^2(x)=x$, for any $x\in A$.  We say that an antiautomorphism (involution) $\varphi:A\longrightarrow A$ is a \emph{graded antiautomorphism (involution)} if $\varphi(A_g)=A_g$, for any $g\in G$. 

If $A$ is an algebra, and $\theta:\textrm{Supp}\, \Gamma \longrightarrow \textrm{Supp}\, \Gamma$ is a bijection, we say that a antiautomorphism (involution) $\varphi$ of $A$ is \emph{$\theta$-homogeneous} if $\varphi(A_g)=A_{\theta(g)}$, for any $g\in \textrm{Supp}\, \Gamma$. We say that $\varphi$ is homogeneous if there exists some $\theta$ such that $\varphi$ is $\theta$-homogeneous.
If $\theta$ is the identity map of $\textrm{Supp}\,\Gamma$, $\theta$-homogeneous antiautomorphism means graded antiautomorphism. If $\theta$ is the map $g\mapsto g^{-1}$, $\theta$-homogeneous antiautomorphism means degree-inverting antiautomorphism.

We now recall basic properties of gradings on $UT_n$. In \cite{VaZai} the authors proved that, up to a graded isomorphism, if the characteristic of $K$ is not 2, any $G$-grading on $UT_n$ is elementary, this means each $E_{i,j}$ is homogeneous and there exists $(g_1, \dots, g_n)\in G^n$ such that $\deg E_{i,j}= g_i^{-1}g_j$. The problem in defining an elementary grading on $UT_n$ by means of this $n$-tuple in $G^n$, it that the isomorphism classes are not uniquely defined by an $n$-tuple. To overcome this problem, one may define an elementary grading from an $(n-1)$-tuple $(g_1, \dots, g_{n-1})\in G^{n-1}$ where $\deg E_{i,i+1}=g_i$. Now each isomorphism class of $G$-gradings on $UT_n$ is uniquely determined by one such $(n-1)$-tuple (see \cite{DiVKoVa}).

\section{Homogeneous automorphisms of upper triangular matrices}

In this section we prove that any homogeneous automorphism of $UT_n$ is a conjugation by an invertible homogeneous matrix of degree $e$. In particular, any homogeneous automorphism of $UT_n$ is a graded automorphism. It is important to remark that it is a particular property of $UT_n$, i.e., it does not hold for any algebra. For instance, if $A_1=K\langle X,Y\,|\, [X,Y]=1 \rangle$ is the Weyl algebra endowed with the $\mathbb{Z}$-grading induced by $\deg X=1$, $\deg Y=-1$, then the automorphism defined by $X\mapsto -Y$, $Y\mapsto X$ is a degree inverting automorphism.

We start with some technical results.

\begin{lemma}\label{inverse}
Let $n\geq 3$ be an integer and let let $1<k<n$. If we write an invertible matrix $P\in UT_n$ as \[P=\left(\begin{array}{ccc} A & u & B\\ 0 & x & v \\ 0 & 0 & C\end{array}\right)\in UT_n,\] with $A\in UT_{k-1}, C\in UT_{n-k}$, $u\in M_{k-1\times 1}(K), v\in M_{1\times n-k}(K)$, $x\in K$ and $B\in M_{k-1\times n-k}(K)$. Then 
\[P^{-1}=\left(\begin{array}{ccc} A^{-1} & -A^{-1}ux^{-1} & A^{-1}(ux^{-1}v-B)C^{-1}\\ 0 & x^{-1} & -x^{-1}vC^{-1} \\ 0 & 0 & C^{-1}\end{array}\right).\]
\end{lemma}

\begin{proof}
    It follows from straightforward computations.
\end{proof}

The next lemma asserts that homogeneous automorphisms of $UT_n$ preserve the neutral component. Although it is stated for $UT_n$, the proof also holds for any unitary algebra in which $1$ is homogeneous.

\begin{lemma}\label{invariant}
    Let $\varphi:UT_n \longrightarrow UT_n$  be a homogeneous automorphism. Then \[\varphi((UT_n)_e)=(UT_n)_e.\]
\end{lemma}

\begin{proof}
Let $g\in G$, such that $\varphi((UT_n)_e)=(UT_n)_g$. Since the identity matrix lies in the neutral component, we have $((UT_n)_e)^2=(UT_n)_e$. As a consequence, 
\[\varphi((UT_n)_e)=\varphi(((UT_n)_e)^2)=\varphi((UT_n)_e)^2=((UT_n)_g)^2\subseteq(UT_n)_{g^2}.\]
This implies that $\varphi((UT_n)_e)\subseteq (UT_n)_g\cap (UT_n)_{g^2}$. Since $\varphi$ is an automorphism, $\varphi((UT_n))_e$ is nonzero and this implies that $(UT_n)_g\cap (UT_n)_{g^2}\neq \{0\}$, from which we obtain $g=g^2$, and $g=e$.  
\end{proof}

If $A$ is an algebra and $a\in A$ is an invertible element, the map $\varphi:A\longrightarrow A$ defined by $\varphi(x)=a^{-1}xa$ is an automorphism of $A$. Automorphisms of this type are called inner automorphisms.  We recall a well known result from \cite{Kezlan} that any automorphism on $UT_n$ is inner.

\begin{theorem}\label{inner}
lf $R$ is any commutative ring with unity, then every $R$-algebra automorphism of $UT_n(R)$ is inner.
\end{theorem}

\begin{lemma}\label{entry}
    Let $\varphi(X)=P^{-1}XP$ be an automorphism of $UT_n$, with $P=(p_{ij})\in UT_n$ an invertible matrix. Then if $k\leq l$, the entry $(k,l)$ of $\varphi(E_{kk})$  is $\dfrac{p_{kl}}{p_{kk}}$.
\end{lemma}

\begin{proof}
    Let us write \[P=\left(\begin{array}{ccc} A & u & B\\ 0 & x & v \\ 0 & 0 & C\end{array}\right)\in UT_n,\] where $A, u, B, x, v, C$ are as in Lemma \ref{inverse}. In particular, $x=p_{kk}$. Using the expression of $P^{-1}$ given in Lemma \ref{inverse}, we have $\varphi(E_{ii})=$
    \begin{align*}
            & = P^{-1}E_{ii}P\\
            &=\begin{pmatrix}A^{-1} & -A^{-1}ux^{-1} & A^{-1}(ux^{-1}v-B)C^{-1}\\ 0 & x^{-1} & -x^{-1}vC^{-1} \\ 0 & 0 & C^{-1}\end{pmatrix}
            \begin{pmatrix}
                0 & 0 & 0\\
                0 & 1 & 0\\
                0 & 0 & 0
            \end{pmatrix}
            \begin{pmatrix} A & u & B\\ 0 & x & v \\ 0 & 0 & C\end{pmatrix}\\
            &= \begin{pmatrix}
            0 & -A^{-1}ux^{-1}  & 0\\
            0 & x^{-1}          & 0\\
            0 & 0               & 0
            \end{pmatrix} \begin{pmatrix} A & u & B\\ 0 & x & v \\ 0 & 0 & C\end{pmatrix}\\
            &= \begin{pmatrix}  0 & -A^{-1}u    & -A^{-1}ux^{-1}v\\
                                0 & 1           & x^{-1}v \\
                                0 & 0           & 0 \end{pmatrix}
    \end{align*}
Now the lemma follows once we observe in the above expression that the $k$-th line of $\varphi(E_{kk})$ is the $k-th$ line of $P$ divided by $x=p_{kk}$. 
\end{proof}

We are now ready to prove the main result of this section. Before that, let us observe that if $P\in UT_n$ is an invertible matrix which is homogeneous of degree $e$, then $\phi(X)=P^{-1}XP$ is a homogeneous automorphism. In fact, it is a graded automorphism. The main result of this section is the converse of this fact.

\begin{theorem}\label{Hom.Aut}

Let $UT_n$ be endowed with a $G$-grading and let $\varphi:UT_n \longrightarrow UT_n$ be a homogeneous automorphism. Then there exists a homogeneous invertible matrix $P\in UT_n$ of degree $e$ such that $\varphi(X)=P^{-1}XP$.
\end{theorem}

\begin{proof}
    First we recall that we may consider only the case where the grading is elementary. In particular, for any $i$, $\deg E_{ii}=e$.
    
    The case $n\leq 2$ is quite simple and follows from direct computations. So let us assume $n\geq 3$ and let $\varphi$ be one such automorphism of $UT_n$. By Theorem \ref{inner} there exists an invertible matrix $P=(p_{ij})\in UT_n$ such that $\varphi(X)=P^{-1}XP$, for all $X\in UT_n$.
    
    Let $i\leq j$, and assume $\deg E_{ij}\neq e$. From Lemma \ref{invariant}, for each $i$, $\varphi(E_{ii})\in (UT_n)_e$. 
    From Lemma $\ref{entry}$, we know that entry $(i,j)$ of $\varphi(E_{ii})$ is $\dfrac{p_{ij}}{p_{ii}}$. Since $\deg E_{ij}\neq e$ and $\varphi(E_{ii})\in (UT_n)_e$, we obtain $p_{ij}=0$. That is, the entries $(i,j)$ of $P$ such that $\deg E_{ij}\neq e$ are zero. From this we conclude that the matrix $P$ has homogeneous degree $e$, and the theorem is proved. 
\end{proof}

Direct consequences now follow.

\begin{corollary}

Every homogeneous automorphism of $UT_n$ is a graded automorphism.

\end{corollary}

\begin{corollary}
    Any graded automorphism of $UT_n$ is given by the conjugation of an invertible homogeneous matrix of degree $e$.
\end{corollary}

\section{Homogeneous antiautomorphisms and involutions of upper triangular matrices}

Let us now study conditions under which an antiautomorphism $\varphi$ of $UT_n$ is homogeneous. The following lemma is stated for $UT_n$ but it holds for any algebra $A$ in which every automorphism is intern.

\begin{lemma}
    Let $\varphi$ and $\eta$ be antiautomorphisms of $UT_n$. Then there exists an invertible matrix $P\in UT_n$ such that $\varphi(X)=P^{-1}\eta(X)P$, for any $X\in UT_n$.
\end{lemma}

\begin{proof}
    First we notice that if $\eta$ is an antiautomorphism of an algebra $A$, its inverse is also an antiautomorphism of $A$. Also, the composition of two antiautomorphism of $A$ is an automorphism of $A$. Hence $\varphi \circ \eta^{-1}$ is an automorphism of $UT_n$. By Theorem  \ref{inner} there exists an invertible matrix $P\in UT_n$ such that $\varphi \circ \eta^{-1}(Y)=P^{-1}YP$, for all $Y\in UT_n$. In particular, for $Y=\eta(X)$, we have $\varphi(X)=P^{-1}\eta(X)P$ and this holds for any $X\in UT_n$ once $\eta$ is bijective.
\end{proof}

Let us denote by $\circ$ the involution on $UT_n$ defined on matrix units as $E_{i,j}^\circ=E_{n-j+1,n-i+1}$. We will call it the \emph{canonical involution} on $UT_n$. Notice that this involution is a reflection along secondary diagonal. For this reason it will also be called \emph{reflection involution}.

From the above lemma, we obtain that if $\varphi:UT_n\longrightarrow UT_n$ is an antiautomorphism, then there exists an invertible matrix $P\in UT_n$ such that $\varphi(X)=P^{-1}X^\circ P$, for all $X\in UT_n$.

A direct application of Lemma \ref{entry} yields the following lemma.

\begin{lemma}
    Let $\varphi(X)=P^{-1}X^\circ P$ be an antiautomorphism of $UT_n$, with $P=(p_{ij})\in UT_n$ an invertible matrix. Then if $k\leq l$, the entry $(k,l)$ of $\varphi(E_{n-k+1, n-k+1})$  is $\dfrac{p_{kl}}{p_{kk}}$.
\end{lemma}

One can observe that if $P\in UT_n$ is an invertible matrix with $\deg P=e$, and the involution $\circ$ of $UT_n$ is homogeneous, then the antiautomorphism $\varphi(X)=P^{-1}X^\circ P$ is also homogeneous. 

The converse of the above fact also holds.

\begin{proposition}
    Let $\varphi:UT_n\longrightarrow UT_n$ be an antiautomorphism of $UT_n$ given by $\varphi(X)=P^{-1}X^\circ P$, for some invertible matrix $P\in UT_n$. Then $\varphi$ is homogeneous if and only if $P$ is homogeneous of degree $e$ and $\circ$ is homogeneous.
\end{proposition}

\begin{proof}
    The proof is a simple adaptation of the proof of Theorem \ref{Hom.Aut}. The only change needed is to apply $\varphi$ in $E_{n-i+1, n-i+1}$ instead of applying it in $E_{i,i}$. 
\end{proof}

The above proposition has a remarkable consequence:

\begin{corollary}
    Let $UT_n$ be endowed with a $G$-grading. Then there exists a homogeneous antiautomorphism of $UT_n$ if and only if the canonical involution $\circ$ of $UT_n$ is homogeneous. Any other homogeneous involution is a composition map of $\circ$ with a graded automorphism of $UT_n$. 
\end{corollary}

By the above, in order to determine the homogeneous antiautomorphisms of $UT_n$, it is enough to find conditions on the grading such that the involution $\circ$ is homogeneous. 

As mentioned before, any group grading on $UT_n$ is elementary and any such isomorphism class is uniquely determined by an $(n-1)$-tuple of elements of $G$, $(g_1, \dots, g_{n-1})$ with the condition $\deg E_{i,i+1}=g_i$, for $i\in\{1, \dots, n-1\}$. So from now on we assume $UT_n$ is endowed with an elementary grading $\Gamma$ defined by the $(n-1)$-tuple $(g_1, \dots, g_{n-1})$ of elements of $G$ such that $\deg E_{i,i+1}=g_i$.

\begin{lemma}
    Let $\Gamma: UT_n=\oplus_{g\in G}A_g$ be an elementary $G$-grading on $UT_n$ defined by the $(n-1)$-tuple $(g_1, \dots, g_{n-1})$, and let us assume $\circ$ is a $\theta$-homogeneous involution on $UT_n$. 
    \begin{enumerate}
        \item If $g\in \textrm{Supp}\, \Gamma$, then $\theta^2(g)=g$.
        \item If $A_gA_h\neq \{0\}$ then $\theta(gh)=\theta(h)\theta(g)$.
    \end{enumerate}
\end{lemma}

\begin{proof}
    Item $(1)$ follows directly from the fact that $\circ$ is an involution.
    
    To prove $(2)$ assume $A_gA_h\neq\{0\}$. Since $\circ$ is bijective, $(A_gA_h)^\circ\neq\{0\}$. 
    On one hand, $A_gA_h\subseteq A_{gh}$ and this implies that \[(A_gA_h)^\circ\subseteq (A_{gh})^\circ=A_{\theta(gh)}.\]
    On the other hand, \[(A_gA_h)^{\circ}=A_h^\circ A_g^\circ=A_{\theta(h)}A_{\theta(g)}\subseteq A_{\theta(h)\theta(g)}.\]
    By the above, \[\{0\}\neq (A_g A_h)^\circ\subseteq A_{\theta(gh)}\cap A_{\theta(h)\theta(g)}, \] 
    and we must have $\theta(gh)=\theta(h)\theta(g)$.
\end{proof}

Let us now analyze the action of $\circ$ on matrix units. 

Before that, we observe that we have a closed expression for $\textrm{Supp}\, \Gamma$.\linebreak Indeed, since $\deg E_{i,i}=e$ and for any $i\in\{1, \dots, n-1\}$ and $k\in \{1, \dots, n-i\}$, $\deg E_{i,i+k}=g_i\cdots g_{i+k-1}$, we have
\[\textrm{Supp}\, \Gamma= \{g_i \cdots g_{i+k-1}\,|\, i\in\{1, \dots, n-1\}, \, k\in \{1, \dots, n-i \}\}\cup\{e\}.\]

We have $E_{i,j}^\circ = E_{n-j+1,n-i+1}$.  In particular, $E_{i,i+1}^\circ = E_{n-i,n-i+1}$ and assuming the map $\circ$ $\theta$-homogeneous, we must have for any $i\leq j$, $\deg(E_{i,j})^\circ = \theta (\deg E_{i,j})$. In particular, for $j=i+1$, we obtain 
\[\theta(g_i)=\theta(\deg E_{i,i+1})=\deg (E_{i.i+1})^\circ= \deg (E_{n-i,n-i+1}) = g_{n-i}.\]
By the homogeneity of $\circ$, if $g_i=g_j$, we obtain $g_{n-i}=g_{n-j}$.
More generally, if $g_i\cdots g_{i+k-1}=g_j \cdots g_{j+l-1}\in \textrm{Supp}\, \Gamma$, the homogeneity of $\circ$ implies $\theta (g_i\cdots g_{i+k-1}) = \theta(g_j \cdots g_{j+l-1})$. By applying the properties of $\theta$, and the expressions of $\theta(g_i)$ above, we obtain that if $g_i\cdots g_{i+k-1}=g_j \cdots g_{j+l-1}\in \textrm{Supp}\, \Gamma$ then $g_{n-i-k+1}\cdots g_{n-i}=g_{n-j-l+1}\cdots g_{n-j}$.

Next we present the main result of the paper, which asserts that the above conditions are also necessary.

\begin{theorem}
    Let $G$ be a group and let $\Gamma: UT_n=\oplus_{g\in G} A_g$ be a $G$-grading on $UT_n$. Then the following assertions are equivalent.
    \begin{enumerate}
        \item The canonical involution $\circ$ on $UT_n$ is homogeneous.
        \item If $g_i\cdots g_{i+k-1}=g_j \cdots g_{j+l-1}\in \textrm{Supp}\, \Gamma$ then we have $g_{n-i-k+1}\cdots g_{n-i}=g_{n-j-l+1}\cdots g_{n-j}$.
    \end{enumerate}
\end{theorem}

\begin{proof}
    The implication (1) $\Longrightarrow$ (2) was discussed above.
    
    So we just need to prove that if (2) is true, then $\circ$ is homogeneous.
    To prove $\circ$ is homogeneous, we must define a bijective map $\theta: \textrm{Supp}\, \Gamma \longrightarrow \textrm{Supp}\, \Gamma$ such that $(A_g)^\circ = A_{\theta(g)}$, for any $g\in \textrm{Supp}\, \Gamma$.
    
    We define $\theta$ on $\{g_1, \dots, g_n\}$ by $\theta(g_i)=g_{n-i}$ and extend it to any element of $\textrm{Supp}\, \Gamma$ by defining $\theta(g_i\cdots g_{i+k-1})=g_{n-i-k+1}\cdots g_{n-i}$.
    
    First we observe that $\theta$ is well defined. Indeed, any element of $\textrm{Supp}\, \Gamma$ is of the form above. Moreover, if one such element is represented in two different ways as $g_i\cdots g_{i+k-1}=g_j \cdots g_{j+l-1}$, condition (2) asserts that their value under $\theta$ is the same.
    
    Now we prove that $\theta$ is bijective. Since $\textrm{Supp}\, \Gamma$ is a finite set, it is enough to show that $\theta$ is injective. Assume that $\theta(g_i\cdots g_{i+k-1})=\theta(g_j\cdots g_{j+l-1})$. This means $g_{n-i-k+1}\cdots g_{n-i} = g_{n-j-l+1}\cdots g_{n-j}$. Again, condition (2) shows that the last equation implies that  $g_i\cdots g_{i+k-1}=g_j\cdots g_{n-l-j}$ proving that $\theta$ is injective.
    
    Finally, we just need to prove that $(A_g)^\circ=A_{\theta(g)}$, for each $g\in \textrm{Supp}\, \Gamma$. This follows directly by applying the map $\circ$ on matrix units. Indeed, since $g=\deg(E_{i,i+k})=g_i\cdots g_{i+k-1}$ and $(E_{i,i+k})^\circ=E_{n-i-k+1,n-i+1}$, we see that $\deg(E_{n-i-k+1,n-i+1})=g_{n-i+k-1}\cdots g_{n-i}=\theta(g)$, and the proof is complete.
\end{proof}

When we restrict to the case of involutions, we have

\begin{corollary}
    Let $UT_n$ be endowed with an elementary $G$-grading defined by the $(n-1)$-tuple $(g_1, \dots, g_{n-1})$ of elements of $G$. Then an involution $\varphi:UT_n \longrightarrow UT_n$ is a $\theta$-homogeneous involution if and only if the following conditions hold.
    \begin{enumerate}
        \item If $g_i\cdots g_{i+k-1}=g_j \cdots g_{j+l-1}\in \textrm{Supp}\, \Gamma$ then we have $g_{n-i-k+1}\cdots g_{n-i}=g_{n-j-l+1}\cdots g_{n-j}$.
        \item There exists an invertible matrix $P\in UT_n$ of degree 0 such that $\varphi(X)=P^{-1}X^\circ P$, for any $X\in UT_n$. The matrix $P$ satisfy $P^\circ=P$ or $P^{\circ}=-P$ if $n$ is even and $P^\circ = P$ if $n$ is odd.
    \end{enumerate} 
\end{corollary}

\begin{proof}
    The only thing we need to prove is that $P$ must satisfy $P^\circ=\pm P$. Actually this is a classical result. We include its proof here for the sake of completeness. Since $\varphi(X)=P^{-1}X^\circ P$ and $\varphi^2$ is the identity map, we obtain $X=P^{-1}P^\circ X (P^{-1})^\circ P$. Using that $(P^{-1})^\circ=(P^{\circ})^{-1}$, we obtain that for any $X\in UT_n$, $P^{-1}P^{\circ}X=XP^{-1}P^{\circ}$, i.e., $P^{-1}P^{\circ}$ is a central element of $UT_n$. Since the central elements of $UT_n$ consist of scalar matrices, it follows that $P^{\circ}=\lambda P$, for some $\lambda \in K$. Since $P^{\circ\circ}=P$, we obtain $P=\lambda^2 P$ which implies that $\lambda=\pm 1$. Finally we observe that the case $\lambda=-1$ cannot occur if $n$ is odd. Indeed, if $n$ is odd, the condition $P^\circ=-P$ would imply that the entry $((n+1)/2, (n+1)/2)$ of $P$ would be zero, which cannot occur once $P$ is invertible.
\end{proof}

Now let us assume that the involution $\circ$ of $UT_n$ is $\theta$-homogeneous. Since an arbitrary homogeneous antiautomorphism (involution) $\varphi$ is a composition map of $\circ$ with a graded automorphism, we have the following corollary.

\begin{corollary}
    Let $UT_n$ be graded by a group $G$. Then an arbitrary homogeneous antiautomorphism (involution) $\varphi$ of $UT_n$ is $\theta$-homogeneous if and only if $\circ$ is $\theta$-homogeneous.
\end{corollary}

The above corollary means that once  a $G$-grading structure on $UT_n$ is given, the type $\theta$ of homogeneous antiautomorphisms (involutions) of $UT_n$ is completely determined. 
When restricted to the more studied cases of homogeneous antiautomorphisms (involutions), we have the following corollary.

\begin{corollary}
    Let $UT_n$ be graded by a group $G$. Then an arbitrary homogeneous antiautomorphism (involution) $\varphi$ of $UT_n$ is 
    \begin{enumerate}
        \item a graded antiautomorphism (involution) if and only if $\circ$ is a graded involution.
        \item a degree-inverting antiautomorphism (involution) if and only if $\circ$ is a degree-inverting involution.
    \end{enumerate}   
\end{corollary}

In particular, we obtain the interesting fact that if a homogeneous antiautomorphism on $UT_n$ is graded, then any homogeneous antiautomorphism of $UT_n$ is graded.

\section*{Funding}

The author was supported by grants \#2018/15627-2 and \#2018/23690-6 S\~ao Paulo Research Foundation (FAPESP).

\end{document}